\documentclass[preprint,11pt]{article}
\usepackage{amssymb,bm}
\usepackage{amsfonts}
\usepackage{mathrsfs}
\usepackage{amssymb,amsfonts,amsmath,mathtools}
\usepackage{graphicx,graphics,epstopdf}
\usepackage{url,bigints}
\usepackage{cleveref}
\graphicspath{{./figures/}}
\usepackage{latexsym,float,tikz,algorithmic}
\usepackage{placeins}
\usepackage[ruled,linesnumbered]{algorithm2e}
\RestyleAlgo{ruled}
\SetKwComment{Comment}{/* }{ */}
\let\oldnl\nl
\newcommand{\nonl}{\renewcommand{\nl}{\let\nl\oldnl}}
\usepackage{adjustbox,accents}
\usepackage{esint}
\usepackage{multirow}
\usepackage{ctable}

\usepackage{subfigure}
\usepackage{times}
\usepackage{multicol,enumerate}
\usepackage{epstopdf}
\usepackage{color}
\usepackage{fullpage}
\usepackage{setspace}
\usepackage[english]{babel}
\usepackage[version=3]{mhchem}
\usepackage{amsmath, amssymb, amsthm}
\usepackage{verbatim, latexsym}
\usepackage{colortbl}
\usepackage{multirow}
\usepackage[OT2,OT1]{fontenc}
\newcommand\cyr
{
\renewcommand\rmdefault{wncyr}
\renewcommand\sfdefault{wncyss}
\renewcommand\encodingdefault{OT2}
\normalfont
\selectfont
}
\DeclareTextFontCommand{\textcyr}{\cyr}

\def\XXint#1#2#3{{\setbox0=\hbox{$#1{#2#3}{\int}$ }
\vcenter{\hbox{$#2#3$ }}\kern-.6\wd0}}

\newtheorem{theorem}{Theorem}[section]

\newtheorem{assumption}{Assumption}[section]

\newtheorem{lemma}{Lemma}[section]

\numberwithin{equation}{section}
\usepackage{soul}
\setstcolor{red}

\title{Edge Multiscale Finite Element Methods}
\author{Shubin Fu\thanks{Eastern Institute for Advanced Study, Eastern Institute of Technology, Ningbo, Zhejiang 315200, P. R. China. (\texttt{sfu@eitech.edu.cn})}\and Guanglian Li\thanks{Corresponding author. Department of Mathematics, The University of Hong Kong, Pokfulam Road, Hong Kong SAR, P.R. China. ({\tt{lotusli@maths.hku.hk}})}}
\begin{document}
\maketitle
\begin{abstract}
The objective of this paper is to review recent developments in Edge Multiscale Finite Element Methods (EMsFEM) for partial differential equations with heterogeneous coefficients or highly oscillatory solutions. Using elliptic equations with heterogeneous coefficients as an illustrative example, we present the key ideas of the method. We also analyze the approach while accounting for the discrete error in the multiscale basis functions. Extensive numerical tests are provided to validate the performance of the method.
\end{abstract}
\section{Introduction}
Partial differential equations (PDEs) with multiscale coefficients arise ubiquitously in applications. For instance, elliptic equations with multiscale permeability coefficients are fundamental to reservoir simulation models. The presence of multiple scales renders standard numerical methods---which typically rely on local polynomial basis functions---computationally infeasible or prohibitively expensive. These challenges have motivated the intensive development of homogenization theory and multiscale numerical methods over the past several decades; see, e.g., \cite{MR2916381,MR4298217,MR4891112,MR4754301,MR4628018,MR2477579,MR4520724,MR4700411,MR4490299,MR4372648,MR4191211,MR3971243,MR3177856}.

This work reviews recent progress in edge multiscale finite element methods (EMsFEMs). First introduced in \cite{MR3939320,MR3980476}, EMsFEMs have been successfully applied to a variety of PDEs with heterogeneous coefficients \cite{MR4265650,MR4343247,MR4860976,WLL2026}. The core idea is to decompose the solution locally over overlapping subdomains and then combine these local representations into a global decomposition via partition of unity functions \cite{melenk1996partition}. Consequently, the global error estimate reduces to estimating the local error within each subdomain, making local error analysis a central component of the theory. Our proof technique is inspired by the transposition method, which provides a priori estimates in weighted $L^2$-norms for homogeneous elliptic equations with nonhomogeneous Dirichlet data \cite{MR0350177}.

The main contributions of this paper are threefold. First, we present a discrete-level formulation of EMsFEMs for general PDEs with heterogeneous coefficients. Second, we provide a concise, structured error analysis. Third, we conduct several numerical tests to validate the performance of EMsFEMs, with particular emphasis on the influence of the overlap size.

The paper is organized as follows. Section~\ref{sec:model} introduces the model problem---an elliptic PDE with multiple scales---and establishes the necessary mathematical framework. In Section~\ref{sec:method-error}, we present the edge multiscale finite element method (EMsFEM) and derive the main error estimates. Numerical experiments that validate the theoretical results and illustrate the capabilities of the method are provided in Section~\ref{sec:num}. We conclude in Section~\ref{sec:conclusion} with a summary of contributions and a discussion of potential directions for future work.

Throughout, we use standard notation for Lebesgue spaces \(L^\infty(G)\) and Sobolev spaces \(W^{m,p}(G)\) on an open bounded domain \(G \subset \mathbb{R}^d\). Norms of \(H^m(G)\) (for \(p=2\)) is denoted by \(\|\cdot\|_{H^m(G)}\). The \(L^2\) scalar product on \(G\) is \((\cdot,\cdot)_G\). The notation $a\lesssim b$ means that there exists a constant $C>0$, independent of the parameters of interest ($h,\bar{d},\delta$), such that $a\leq Cb$. 
\section{Model problem}\label{sec:model}
Let $\Omega \subset \mathbb{R}^d$ be a bounded Lipschitz domain with boundary $\Gamma \coloneqq \partial\Omega$. We consider the boundary value problem
\begin{equation}\label{eq:model}
\left\{
\begin{aligned}
\mathcal{A}u &= f &&\text{in } \Omega,\\
\mathcal{B}_j u &= g_j &&\text{on } \Gamma,\; j=0,\dots,m-1.
\end{aligned}
\right.
\end{equation}

For simplicity, we focus on the case where $\mathcal{A}: H^1(\Omega) \to H^{-1}(\Omega)$ is a second-order partial differential operator of the form
\begin{equation}\label{eq:pde-operator}
\mathcal{A}u = \sum_{|\bm{p}|,|\bm{q}|\le 1}(-1)^{|\bm{p}|}
\frac{\partial^{\bm{p}}}{\partial \bm{x}^{\bm{p}}}
\Bigl( a_{\bm{p},\bm{q}}(\bm{x})\,
\frac{\partial^{\bm{q}}u}{\partial \bm{x}^{\bm{q}}}\Bigr),
\end{equation}
and the boundary operator $\mathcal{B}$ (of order $1$) is defined by
\begin{equation}\label{eq:boundary-operator}
\mathcal{B}u = \sum_{|\bm{p}|\le 1} b_{\bm{p}}(\bm{x})\,
\frac{\partial^{\bm{p}}u}{\partial \bm{x}^{\bm{p}}}.
\end{equation}
Here $a_{\bm{p},\bm{q}} \in L^{\infty}(\Omega)$ for all $|\bm{p}|,|\bm{q}|\le 1$, and the coefficients may contain multiple scales that propagate into the solution $u$. The boundary coefficients are taken as
\[
b_{\bm{p}} \coloneqq \sum_{i=1}^d a_{\bm{e}_i,\bm{p}},
\]
where $\{\bm{e}_i\}_{i=1}^d$ are the canonical basis vectors of $\mathbb{R}^d$. We assume $f \in L^2(\Omega)$ and $g \in H^{-1/2}(\Gamma)$. Such multiscale coefficients arise naturally in many applications, for instance in composite materials and porous media.

Problems of the form \eqref{eq:model} cover a wide range of situations, including uniformly elliptic operators with rough coefficients, Helmholtz operators with large wavenumber, convection-dominated diffusion equations, and indefinite Maxwell equations with large wavenumber (see, e.g., \cite{MR3980476,MR4265650,MR4343247,MR4860976,WLL2026}). Moreover, Edge Multiscale Finite Element Methods (EMsFEMs) are not restricted to stationary linear operators; they have been adapted, for example, to semilinear parabolic problems with heterogeneous coefficients \cite{poveda2024edge}.

The weak formulation of \eqref{eq:model}, with $\mathcal{A}$ and $\mathcal{B}$ given by \eqref{eq:pde-operator} and \eqref{eq:boundary-operator}, reads: find $u \in H^{1}(\Omega)$ such that
\begin{equation}\label{var_form}
a(u,v) = (f,v)_{\Omega} + \langle g, v\rangle_{\Gamma}
\qquad\forall v \in H^{1}(\Omega),
\end{equation}
where the bilinear form $a(\cdot,\cdot): H^1(\Omega)\times H^1(\Omega) \to \mathbb{R}$ is
\[
a(v,w) \coloneqq 
\sum_{|\bm{p}|,|\bm{q}|\le 1} 
\int_{\Omega} a_{\bm{p},\bm{q}}(\bm{x})\,
\frac{\partial^{\bm{q}}v}{\partial \bm{x}^{\bm{q}}}\,
\frac{\partial^{\bm{p}}w}{\partial \bm{x}^{\bm{p}}} \,\mathrm{d}\bm{x},
\qquad v,w\in H^1(\Omega).
\]
$\langle \cdot,\cdot,\rangle_{\Gamma}$ denote the duality paring between $H^{-1/2}(\Gamma)$ and $H^{1/2}(\Gamma)$. 

We assume that $a(\cdot,\cdot)$ is coercive and bounded, i.e., there exist constants $0<\alpha\le\beta<\infty$ such that
\begin{align*}
a(v,v) &\ge \alpha\, \|v\|_{H^1(\Omega)}^2
&&\forall v\in H^1(\Omega),\\
a(v,w) &\le \beta\, \|v\|_{H^1(\Omega)}\|w\|_{H^1(\Omega)}
&&\forall v,w\in H^1(\Omega).
\end{align*}
Under these conditions, the Lax--Milgram theorem guarantees the well-posedness of \eqref{var_form}.

\section{Methodology and error bound}\label{sec:method-error}
We present in this section the edge multiscale finite element method to approximate the weak solution $u\in H^1(\Omega)$ defined in \eqref{var_form}.
\subsection{Notation and preliminaries}
Let $\mathcal{T}_H := \{K_i \mid i=1,\dots,N\}$ be a quasi-uniform partition of $\Omega$ into coarse elements with mesh size $H$, and let $\mathcal{T}_h$ be a refinement of $\mathcal{T}_H$. For each coarse element $K_i \in \mathcal{T}_H$, define the overlapping subdomain
\[
\Omega_i := \bigcup\bigl\{T \in \mathcal{T}_h : \operatorname{dist}(T,K_i) \le \delta_i\bigr\},
\]
and set $\delta := \min_{K_i\in\mathcal{T}_H}\delta_i$. The collection $\{\Omega_i\}_{i=1}^N$ forms an overlapping cover of $\Omega$, i.e. $\overline{\Omega} = \bigcup_{i=1}^N\overline{\Omega}_i$. Denote $d_i := \operatorname{diam}(\Omega_i)$ and $\bar{d} := \max\{d_i : i=1,\dots,N\}$. By construction,
\[
d_i \le H + 2\delta_i \qquad (i=1,\dots,N).
\]

We make the \emph{finite-overlap assumption}: there exists a constant $\Lambda>1$, independent of $N$, such that
\begin{equation}\label{eq:finite-overlap}
\Lambda = \max_{i=1,\dots,N} \#\Lambda_i, \qquad 
\Lambda_i := \bigl\{j : \overline{\Omega}_i \cap \overline{\Omega}_j \neq \emptyset\bigr\}.
\end{equation}

Let $V_h \subset H^1(\Omega)$ be the standard $P_1$ finite element space on the fine mesh $\mathcal{T}_h$, i.e.
\[
V_h := \bigl\{ v \in H^1(\Omega) : v|_T \in \mathcal{P}_1(T) \text{ for all } T \in \mathcal{T}_h\bigr\},
\]
then the Galerkin approximation of \eqref{var_form} reads: find $u_h \in V_h$ such that
\begin{equation}\label{eq:discrete-global}
a(u_h,v_h) = (f,v_h)_{\Omega} + \langle g, v_h\rangle_{\Gamma} \qquad \forall v_h \in V_h.
\end{equation}

For each subdomain we introduce the local spaces
\[
V_{h,i} := \bigl\{v_h|_{\overline{\Omega}_i} : v_h \in V_h\bigr\}, \qquad 
V_{h,i}^0 := \bigl\{ v_h \in V_{h,i} : v_h = 0 \text{ on } \partial\Omega_i\bigr\}.
\]

Let $\{x_j : j \in \mathcal{I}\}$ be the set of vertices of $\mathcal{T}_h$ in $\overline{\Omega}$, and $\mathcal{I}_i \subset \mathcal{I}$ those belonging to $\overline{\Omega}_i$. The prolongation $R_i^{\,T} : V_{h,i} \to V_h$ is defined nodewise by
\[
R_i^{\,T} v_i (x_j) = 
\begin{cases}
v_i(x_j), & j \in \mathcal{I}_i,\\
0, & \text{otherwise}.
\end{cases}
\]
Note that $R_i^{\,T} v_i$ is the $H^1$-conforming finite element extension by zero of $v_i$ (the exact zero extension is generally not in $H^1(\Omega)$).

To obtain a more stable prolongation we use a partition of unity. Let $\{\chi_i\}_{i=1}^N \subset V_h$ be a family of nonnegative functions satisfying
\begin{equation}\label{eq:pum-property}
\begin{cases}
\displaystyle \sum_{i=1}^N \chi_i = 1 \text{ in } \overline{\Omega}, \quad
\operatorname{supp}(\chi_i) \subset \overline{\Omega}_i,\\[4pt]
0 \le \chi_i \le 1, \quad 
\|\nabla \chi_i\|_{L^\infty(\Omega_i)} \le C_{\mathrm{G}}\, \delta_i^{-1},
\end{cases}
\end{equation}
with a constant $C_{\mathrm{G}}$ independent of $h,\bar{d},\delta$. The weighted prolongation $\widetilde{R}_i^{\,T} : V_{h,i} \to V_h$ is then defined by
\[
\widetilde{R}_i^{\,T} v_i := R_i^{\,T}\bigl(\Pi_{h,i}(\chi_i v_i)\bigr),
\]
where $\Pi_{h,i} : C(\overline{\Omega}_i) \to V_{h,i}$ denotes the local nodal interpolation. The operator $\Pi_{h,i}$ satisfies the stability estimate (see e.g. \cite{MR520174})
\begin{equation}\label{eq:nodal-interp}
\|(I-\Pi_{h,i})v_i\|_{L^2(\Omega_i)} + h \|\nabla (I-\Pi_{h,i})v_i\|_{L^2(\Omega_i)}
\le C_I h \|\nabla v_i\|_{L^2(\Omega_i)} \quad \forall v_i \in V_{h,i},
\end{equation}
with $C_I$ depending only on the shape regularity of $\mathcal{T}_h$. By construction,
\begin{equation}\label{eq:pou}
\sum_{i=1}^N \widetilde{R}_i^{\,T}\bigl(v_h|_{\Omega_i}\bigr) = v_h \qquad \forall v_h \in V_h.
\end{equation}
From \eqref{eq:finite-overlap} we also have, for any $v_i \in H^1(\Omega_i)$,
\begin{equation}\label{eq:sumpou-ieq}
\begin{aligned}
\Bigl\|\sum_{i=1}^N \chi_i v_i\Bigr\|_{L^2(\Omega)}^2 &\le \Lambda \sum_{i=1}^N \|\chi_i v_i\|_{L^2(\Omega_i)}^2,\\
\Bigl\|\sum_{i=1}^N \nabla(\chi_i v_i)\Bigr\|_{L^2(\Omega)}^2 &\le \Lambda \sum_{i=1}^N \|\nabla(\chi_i v_i)\|_{L^2(\Omega_i)}^2.
\end{aligned}
\end{equation}

\subsection{Edge multiscale ansatz space}\label{subsec:edge-ansatz}

We now define the edge multiscale approximation space $V_{\mathrm{ms},\ell}$. For simplicity we describe the construction for $d=2$; the extension to $d=3$ is straightforward.

Fix a level parameter $\ell \in \mathbb{N}$. Let $\{\Gamma_i^k\}_{k=1}^{n_i}$ be a partition of $\partial\Omega_i$ into non-overlapping coarse edges. On each edge $\Gamma_i^k$ we consider the space $V_{i,\ell}^k \subset C(\partial\Omega_i)$ spanned by hierarchical bases up to level $\ell$, continuous across the whole boundary. The local edge space on $\partial\Omega_i$ is the smallest linear space containing all $V_{i,\ell}^k$, i.e.
\begin{equation}\label{eq:local-edge}
V_{i,\ell} := \operatorname{span}\bigl\{\psi_{i,\ell}^j : 1 \le j \le n_i 2^{\ell}\bigr\},
\end{equation}
where $\{\psi_{i,\ell}^j\}_{j=1}^{n_i2^{\ell}}$ are the corresponding nodal basis functions and $\{x_{i,\ell}^j\}_{j=1}^{n_i2^{\ell}}$ the associated nodes. Figure~\ref{fig:grid-points-local} illustrates the distribution of these nodes for $n_i=4$ and $\ell = 0,1,2$.

\begin{figure}[htbp]
\centering
\begin{tikzpicture}[scale=1]
	\draw[step=1.0, black, very thick] (0,0) grid (4,4);
	\foreach \x in {0,...,4}
	\foreach \y in {0,...,4}{
		\fill (1.0*\x, 1.0*\y) circle (1.5pt);
	}
	\fill [red] (2.0 , 2.0) circle (1.5pt);
	\filldraw[draw=black,fill=green,opacity=0.4] (1.8,1.8) rectangle (3.2,3.2);
	\node at (1.8,1.8) {$\Omega_i$};
	
	\filldraw[draw=black,fill=green,opacity=0.4] (4.8,3.8) rectangle (6.2,5.2);
	\foreach \x in {4.8,6.2}
	\foreach \y in {3.8,5.2}{
		\fill [blue] (\x, \y) circle (1.5pt);
	}
	
	\filldraw[draw=black,fill=green,opacity=0.4] (4.8,1.8) rectangle (6.2,3.2);
	\foreach \x in {4.8,5.5,6.2}
	\foreach \y in {1.8,3.2}{
		\fill [blue] (\x, \y) circle (1.5pt);
	}
	\foreach \x in {4.8,6.2}
	\foreach \y in {2.5}{
		\fill [blue] (\x, \y) circle (1.5pt);
	}
	
	\filldraw[draw=black,fill=green,opacity=0.4] (4.8,-0.2) rectangle (6.2,1.2);
	\foreach \x in {4.8,5.15,5.5,5.85,6.2}
	\foreach \y in {-0.2,1.2}{
		\fill [blue] (\x, \y) circle (1.5pt);
	}
	\foreach \x in {4.8,6.2}
	\foreach \y in {-0.2,0.15,0.5,0.85,1.2}{
		\fill [blue] (\x, \y) circle (1.5pt);
	}
	
	\draw [-to] (5,4.6) -- (3.0,2.6);
	\draw [-to] (5,2.1) -- (3.0,2.1);
	\draw [-to] (5,0.6) -- (3.0,1.9);
	
	\node at (8.5,4.5) {$\ell=0$};
	\node at (8.5,2.5) {$\ell=1$};
	\node at (8.5,0.5) {$\ell=2$};
\end{tikzpicture}
\caption{Grid points $\{x_{i,\ell}^j\}$ on $\partial\Omega_i$ for $n_i=4$ and levels $\ell=0,1,2$.}
\label{fig:grid-points-local}
\end{figure}
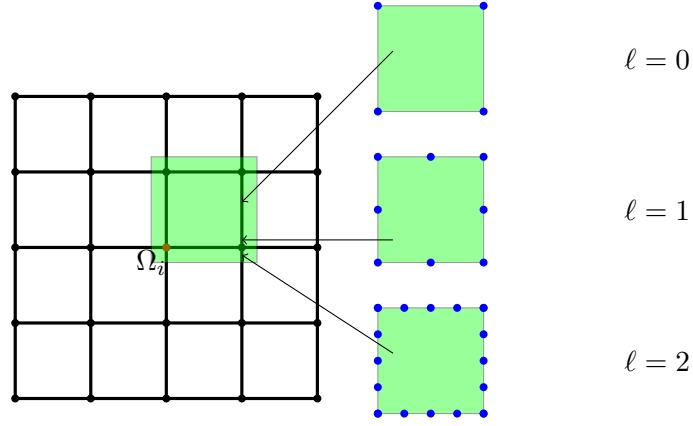

Introduce the local bilinear form on $\Omega_i$,
\[
a_i(v,w) := \sum_{|\bm{p}|,|\bm{q}|\le 1} \int_{\Omega_i} a_{\bm{p},\bm{q}}\,
\frac{\partial^{\bm{q}}v}{\partial\bm{x}^{\bm{q}}}\,
\frac{\partial^{\bm{p}}w}{\partial\bm{x}^{\bm{p}}}\,\mathrm{d}\bm{x},
\qquad v,w \in H^1(\Omega_i).
\]
The local multiscale space on $\Omega_i$ is defined as
\begin{equation}\label{eq:local-multiscale}
V_{\mathrm{ms},\ell}^i := V_{\mathrm{ms},\ell}^{i;\partial} \oplus V_{\mathrm{ms},\ell}^{i;0},
\end{equation}
where
\[
V_{\mathrm{ms},\ell}^{i;\partial} := \operatorname{span}\bigl\{
\mathcal{A}_{h,i}^{-1}(\psi_{i,\ell}^j) : 1 \le j \le n_i 2^{\ell}\bigr\},
\qquad
V_{\mathrm{ms},\ell}^{i;0} := \operatorname{span}\{b_{h,i}\}.
\]
Here $\mathcal{A}_{h,i}^{-1}(\psi_{i,\ell}^j) \in V_{h,i}$ is the discrete generalized harmonic extension of $\psi_{i,\ell}^j$, i.e. the solution of
\begin{equation}\label{eq:Li}
\begin{cases}
a_i(v_{h,i}, w_{h,i}) = 0 & \forall w_{h,i} \in V_{h,i}^0,\\
v_{h,i} = \psi_{i,\ell}^j & \text{on } \partial\Omega_i,
\end{cases}
\end{equation}
and $b_{h,i} \in V_{h,i}^0$ satisfies
\begin{equation}\label{eq:Li-bubble}
a_i(b_{h,i}, w_{h,i}) = 1 \qquad \forall w_{h,i} \in V_{h,i}^0.
\end{equation}
The dimension of $V_{\mathrm{ms},\ell}^i$ is $n_i 2^{\ell}+1$.

Finally, the global edge multiscale space is obtained by applying the weighted prolongation:
\begin{equation}\label{eq:global-multiscale}
V_{\mathrm{ms},\ell} := V_{\mathrm{ms},\ell}^{\partial} \oplus V_{\mathrm{ms},\ell}^0,
\end{equation}
with
\[
V_{\mathrm{ms},\ell}^{\partial} := \bigoplus_{i=1}^N \widetilde{R}_i^{\,T}\bigl(V_{\mathrm{ms},\ell}^{i;\partial}\bigr),
\qquad
V_{\mathrm{ms},\ell}^{0} := \bigoplus_{i=1}^N \widetilde{R}_i^{\,T}\bigl(V_{\mathrm{ms},\ell}^{i;0}\bigr).
\]
By construction, $V_{\mathrm{ms},\ell} \subset V_h$.

\subsection{Conforming Galerkin approximation}\label{sec:edge}

The Edge Multiscale Finite Element Method (EMsFEM) reads: find $u_{\mathrm{ms},\ell} \in V_{\mathrm{ms},\ell}$ such that
\begin{equation}\label{eqn:weakform_h-romb}
a(u_{\mathrm{ms},\ell}, v_{\mathrm{ms},\ell}) = (f, v_{\mathrm{ms},\ell})_{\Omega} + \langle g, v_{\mathrm{ms},\ell}\rangle_{\Gamma}
\qquad \forall v_{\mathrm{ms},\ell} \in V_{\mathrm{ms},\ell}.
\end{equation}
The space $V_{\mathrm{ms},\ell}$ is independent of the source term $f$, which makes the method particularly efficient when multiple right-hand sides are considered, e.g. in uncertainty quantification.

The complete procedure is summarised in Algorithm~\ref{algorithm:wavelet}.

\begin{algorithm}[htbp]
\caption{Edge Multiscale Finite Element Method (EMsFEM)}
\label{algorithm:wavelet}
\begin{algorithmic}[1]
\REQUIRE Level parameter $\ell \in \mathbb{N}$; coarse subdomains $\Omega_i$ and coarse edges $\{\Gamma_i^k\}_{k=1}^{n_i}$; edge spaces $V_{i,\ell}^k$ up to level $\ell$ on each $\Gamma_i^k$.
\ENSURE Approximation $u_{\mathrm{ms},\ell}$.
\STATE Construct the local edge space $V_{i,\ell}$ via \eqref{eq:local-edge}.
\STATE Compute the local multiscale spaces $V_{\mathrm{ms},\ell}^i$ via \eqref{eq:local-multiscale}.
\STATE Assemble the global multiscale space $V_{\mathrm{ms},\ell}$ via \eqref{eq:global-multiscale}.
\STATE Solve the Galerkin problem \eqref{eqn:weakform_h-romb} for $u_{\mathrm{ms},\ell}$.
\end{algorithmic}
\end{algorithm}

\subsection{Error analysis}\label{subsec:error}

To analyse the convergence of Algorithm~\ref{algorithm:wavelet}, we first define a projection operator onto the edge multiscale space. Let $\mathcal{I}_{i,\ell} : L^2(\partial\Omega_i) \to V_{i,\ell}$ be the $L^2$-orthogonal projection, and set
\begin{equation}\label{eq:projectionEDGE}
\mathcal{P}_{i,\ell} v := \mathcal{A}_{h,i}^{-1}\bigl(\mathcal{I}_{i,\ell} v\bigr),
\qquad v \in V_{h,i}(\partial\Omega_i).
\end{equation}
Note that $\mathcal{P}_{i,\ell} v|_{\partial\Omega_i} = \mathcal{I}_{i,\ell} v$. Using the partition-of-unity representation \eqref{eq:pou}, the global projection $\mathcal{P}_\ell : V_h \to V_{\mathrm{ms},\ell}^{\partial}$ is defined by
\begin{equation}\label{eq:glo-proj}
\mathcal{P}_{\ell}(v) := \sum_{i=1}^N \widetilde{R}_i^{\,T}\bigl(\mathcal{P}_{i,\ell}(v|_{\partial\Omega_i})\bigr).
\end{equation}

Next we split the fine-scale solution $u_h$ locally. On each $\Omega_i$ write $u_h|_{\Omega_i} = u_{h,i}^0 + u_{h,i}^{\partial}$, where $u_{h,i}^0 \in V_{h,i}^0$ satisfies
\begin{equation}\label{eq:local-bubble}
a_i(u_{h,i}^0, v_h) = (f, v_h)_{\Omega_i} \qquad \forall v_h \in V_{h,i}^0,
\end{equation}
and $u_{h,i}^{\partial} := u_h|_{\Omega_i} - u_{h,i}^0$ is discrete harmonic, i.e.
\[
a_i(u_{h,i}^{\partial}, v_h) = 0 \qquad \forall v_h \in V_{h,i}^0.
\]
Standard energy arguments together with Friedrichs' inequality give
\begin{align}
d_i \|\nabla u_{h,i}^0\|_{L^2(\Omega_i)} + \|u_{h,i}^0\|_{L^2(\Omega_i)} &\lesssim d_i^2 \|f\|_{L^2(\Omega_i)}, \label{eq:bubble-est}\\
\|u_{h,i}^{\partial}\|_{L^2(\Omega_i)} &\lesssim d_i^2 \|f\|_{L^2(\Omega_i)} + d_i \|u_h\|_{L^2(\Omega_i)}, \label{eq:harm-L2}\\
\|\nabla u_{h,i}^{\partial}\|_{L^2(\Omega_i)} &\lesssim d_i \|f\|_{L^2(\Omega_i)} + \|\nabla u_h\|_{L^2(\Omega_i)}. \label{eq:harm-H1}
\end{align}

\begin{lemma}\label{lem:loc-bubble}
Let $u_{h,i}^0$ and $b_{h,i}$ be defined by \eqref{eq:local-bubble} and \eqref{eq:Li-bubble}, respectively. Then
\[
\min_{c\in\mathbb{R}} \|u_{h,i}^0 - c\,b_{h,i}\|_{H^1(\Omega_i)} \lesssim d_i \|f\|_{L^2(\Omega_i)}.
\]
\end{lemma}

The following assumption is essential for the analysis of discrete harmonic functions.

\begin{assumption}\label{as:veryweak}
For any $v_{h,i} \in V_{h,i}$ satisfying $a_i(v_{h,i}, w_h)=0$ for all $w_h \in V_{h,i}^0$, we assume
\begin{align}
\|v_{h,i}\|_{L^2(\Omega_i)} &\lesssim d_i^{1/2} \|v_{h,i}|_{\partial\Omega_i}\|_{L^2(\partial\Omega_i)}, \label{eq:apriori-L2}\\
\|\chi_i \nabla v_{h,i}\|_{L^2(\Omega_i)} &\lesssim \delta_i^{-1} \|v_{h,i}\|_{L^2(\Omega_i)}. \label{eq:caccioppoli}
\end{align}
\end{assumption}
Inequality \eqref{eq:apriori-L2} is an $L^2$ a priori estimate for discrete harmonic functions, while \eqref{eq:caccioppoli} is a Caccioppoli-type estimate. Assumption~\ref{as:veryweak} has been verified for Helmholtz operators and indefinite convection-diffusion problems \cite{fu2024edge,gl2026conv-diff}, and similar results are available at the PDE level for a wide class of operators \cite{MR3980476,MR4265650,MR4343247,MR4860976,WLL2026,MR5060905,MR3939320}.

\begin{theorem}\label{thm:error}
Let Assumption~\ref{as:veryweak} hold, and let $u_h \in V_h$ and $u_{\mathrm{ms},\ell} \in V_{\mathrm{ms},\ell}$ be the solutions of \eqref{eq:discrete-global} and \eqref{eqn:weakform_h-romb}, respectively. Then
\[
\|u_h - u_{\mathrm{ms},\ell}\|_{H^1(\Omega)} \lesssim \bar{d}
\Bigl( 1 + \frac{\bar{d}}{\delta} + \frac{ 2^{-\ell/2}}{\delta} \Bigr) \|f\|_{L^2(\Omega)}.
\]
\end{theorem}
\begin{proof}
Define the local average $\bar{f}_i := |\Omega_i|^{-1}\int_{\Omega_i} f\,\mathrm{d}\bm{x}$ and set 
$\widehat{u}_{h,i}^0 := \bar{f}_i b_{h,i}$. Let $e_{h,i}^0 := u_{h,i}^0 - \widehat{u}_{h,i}^0$. 
By Lemma~\ref{lem:loc-bubble} and the triangle inequality,
\[
\|e_{h,i}^0\|_{H^1(\Omega_i)} \lesssim d_i \|f\|_{L^2(\Omega_i)}.
\]
More precisely, using Friedrichs' inequality and the coercivity of $a_i(\cdot,\cdot)$ yields
\begin{equation}\label{eq:loc-est1}
d_i \|\nabla e_{h,i}^0\|_{L^2(\Omega_i)} + \|e_{h,i}^0\|_{L^2(\Omega_i)} \lesssim d_i^2 \|f\|_{L^2(\Omega_i)}.
\end{equation}

Next, define $\widehat{u}_{h,i}^{\partial} := \mathcal{P}_{i,\ell}(u_h|_{\partial\Omega_i})$ and let 
$e_{h,i}^{\partial} := u_{h,i}^{\partial} - \widehat{u}_{h,i}^{\partial}$. Since $e_{h,i}^{\partial}$ is discrete harmonic 
and satisfies $e_{h,i}^{\partial}|_{\partial\Omega_i} = (u_h - \mathcal{I}_{i,\ell} u_h)|_{\partial\Omega_i}$, 
Assumption~\ref{as:veryweak} gives
\begin{equation}\label{eq:loc-est2}
\begin{aligned}
\|e_{h,i}^{\partial}\|_{L^2(\Omega_i)} &\lesssim d_i^{1/2} \|u_h - \mathcal{I}_{i,\ell} u_h\|_{L^2(\partial\Omega_i)}, \\
\|\chi_i \nabla e_{h,i}^{\partial}\|_{L^2(\Omega_i)} &\lesssim \delta_i^{-1} \|e_{h,i}^{\partial}\|_{L^2(\Omega_i)}.
\end{aligned}
\end{equation}

Using the approximation property of the $L^2$-projection $\mathcal{I}_{i,\ell}$ on the boundary 
(which follows from the wavelet construction with $d_i2^{-\ell}$ representing the mesh size on $\partial\Omega_i$), we have
\[
\|u_h - \mathcal{I}_{i,\ell} u_h\|_{L^2(\partial\Omega_i)} \lesssim d_i^{1/2}2^{-\ell/2} \|u_h\|_{H^{1/2}(\partial\Omega_i)} 
\lesssim d_i^{1/2}2^{-\ell/2} \|u_h\|_{H^1(\Omega_i)}.
\]

From the coercivity of $a(\cdot,\cdot)$, the inclusion $V_{\mathrm{ms},\ell} \subset V_h$, and Galerkin orthogonality, 
we obtain the best-approximation estimate
\[
\|u_h - u_{\mathrm{ms},\ell}\|_{H^1(\Omega)} \lesssim \min_{v \in V_{\mathrm{ms},\ell}} \|u_h - v\|_{H^1(\Omega)}.
\]

Choosing $v = \sum_{i=1}^N \widetilde{R}_i^{\,T}(\widehat{u}_{h,i}^0 + \widehat{u}_{h,i}^{\partial}) \in V_{\mathrm{ms},\ell}$ yields
\[
\|u_h - u_{\mathrm{ms},\ell}\|_{H^1(\Omega)} \lesssim
\Bigl\| u_h - \sum_{i=1}^N \widetilde{R}_i^{\,T}(\widehat{u}_{h,i}^0 + \widehat{u}_{h,i}^{\partial}) \Bigr\|_{H^1(\Omega)}.
\]

Using the partition-of-unity property \eqref{eq:pou}, $u_h = \sum_{i=1}^N \widetilde{R}_i^{\,T}(u_h|_{\Omega_i})$, 
together with the triangle inequality, we get
\[
\|u_h - u_{\mathrm{ms},\ell}\|_{H^1(\Omega)} \lesssim
\Bigl\| \sum_{i=1}^N \widetilde{R}_i^{\,T}(e_{h,i}^0) \Bigr\|_{H^1(\Omega)}
+ \Bigl\| \sum_{i=1}^N \widetilde{R}_i^{\,T}(e_{h,i}^{\partial}) \Bigr\|_{H^1(\Omega)}.
\]

For the first term, note that $\widetilde{R}_i^{\,T}(e_{h,i}^0) = \chi_i e_{h,i}^0$ up to interpolation error 
(which is of higher order). Using the finite overlap \eqref{eq:finite-overlap} and the stability of the partition of unity,
\[
\Bigl\| \sum_{i=1}^N \widetilde{R}_i^{\,T}(e_{h,i}^0) \Bigr\|_{H^1(\Omega)}^2 \lesssim 
\sum_{i=1}^N \|\chi_i e_{h,i}^0\|_{H^1(\Omega_i)}^2.
\]

Combined with \eqref{eq:loc-est1} and the bound $\|\nabla\chi_i\|_{L^\infty(\Omega_i)} \lesssim \delta_i^{-1}$, this gives
\[
\Bigl\| \sum_{i=1}^N \widetilde{R}_i^{\,T}(e_{h,i}^0) \Bigr\|_{H^1(\Omega)} \lesssim 
\Bigl( \bar{d} + \frac{\bar{d}^2}{\delta} \Bigr) \|f\|_{L^2(\Omega)}.
\]

For the second term, a similar argument using \eqref{eq:loc-est2} and the approximation estimate for $\mathcal{I}_{i,\ell}$ yields
\[
\Bigl\| \sum_{i=1}^N \widetilde{R}_i^{\,T}(e_{h,i}^{\partial}) \Bigr\|_{H^1(\Omega)} \lesssim 
\frac{\bar{d} 2^{-\ell/2}}{\delta} \|u_h\|_{H^1(\Omega)}.
\]

Since $\|u_h\|_{H^1(\Omega)} \lesssim \|f\|_{L^2(\Omega)}$ by the well-posedness of \eqref{eq:discrete-global}, 
we finally obtain
\[
\|u_h - u_{\mathrm{ms},\ell}\|_{H^1(\Omega)} \lesssim
\Bigl( \bar{d} + \frac{\bar{d}^2}{\delta} + \frac{\bar{d} 2^{-\ell/2}}{\delta} \Bigr) \|f\|_{L^2(\Omega)},
\]
which completes the proof.
\end{proof}

\section{Numerical experiments}\label{sec:num}

We present several numerical examples to demonstrate the performance of Algorithm \ref{algorithm:wavelet}. We use $u_h$ defined in \eqref{eq:discrete-global} as the reference solution. For all square coarse subdomains in these two-dimensional tests, we have $n_i=4$ coarse edges on $\partial\Omega_i$. For the Darcy and convection-diffusion tests, the three level parameters are denoted by $\ell=0,1,2$, so the numbers of local boundary basis functions are $n_i 2^{\ell}=4,8,16$, respectively. Equivalently, these are the numbers of local discrete harmonic extensions before adding the bubble function. In all tests, the overlap is uniform, i.e. $\delta_i=\delta$ for all $i$. The common overlap width is varied as
\[
\delta/h = 1,\ldots,16.
\]
The partition-of-unity functions are constructed with cubic smooth transition weights and are normalized nodally, so that \eqref{eq:pum-property} is satisfied. The horizontal axis in the following figures is therefore the common overlap width measured in fine-grid layers. The corresponding coarse dimensions for $\ell=0,1,2$ in the first two tests are $1280$, $2304$, and $4352$.
For the Helmholtz test, only the richest level $\ell=2$ with $n_i 2^{\ell}=16$ is reported, because the fine-grid Helmholtz solution already contains significant phase dispersion at the tested wave number.

For each problem, we also show a representative field comparison using $\ell=2$ (hence $n_i 2^{\ell}=16$) and $\delta/h=2$. In these comparisons the first two panels use the same color scale so that the reference and multiscale fields can be compared directly; the third panel shows the pointwise absolute difference.

\subsection{Darcy problem with a highly oscillatory coefficient}

The first test is the elliptic Darcy problem
\[
-\nabla\cdot(a(x)\nabla u)=1 \quad \text{in } \Omega=(0,1)^2,
\qquad
u=0 \quad \text{on } \partial\Omega .
\]
The fine grid contains $256^2$ elements and the coarse grid contains $16\times16$ blocks; hence $h=1/256$, $H=1/16$, and each coarse block contains $16\times16$ fine elements. The coefficient $a(x)$ is shown in Figure~\ref{fig:darcy-coeff-combined}. Its minimum and maximum values are $1.000\times10^{0}$ and $1.000\times10^{4}$, respectively, so the contrast is $1.000\times10^{4}$.

\begin{figure}[htbp]
\centering
\includegraphics[width=0.58\linewidth]{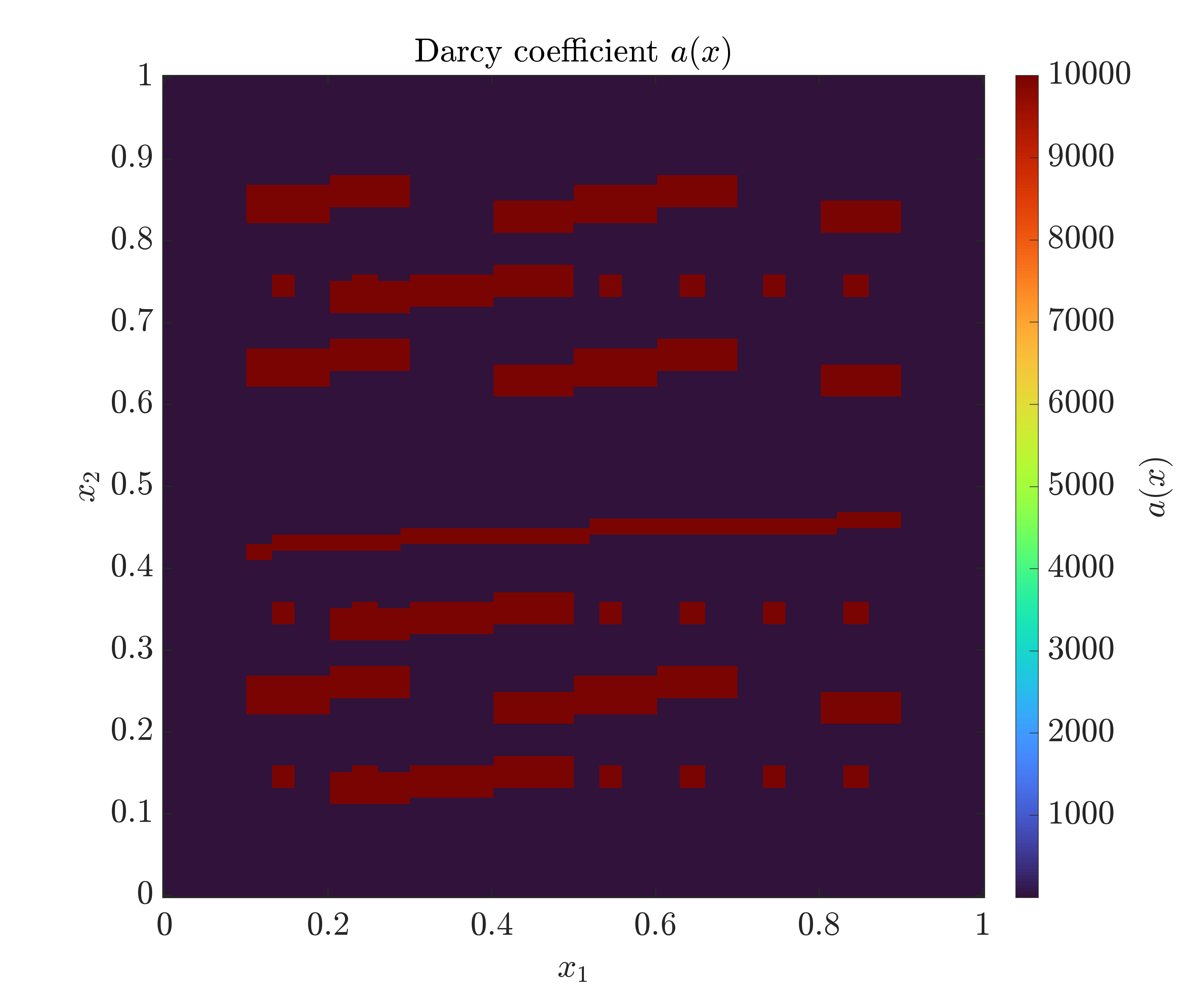}
\caption{Highly oscillatory Darcy coefficient $a(x)$ on the fine grid.}
\label{fig:darcy-coeff-combined}
\end{figure}

For this problem the relative errors are measured in the coefficient-weighted $L^2$ norm and the energy norm,
\[
e_{L^2}=\frac{\|u_h-u_{\mathrm{ms},\ell}\|_{L^2_a}}{\|u_h\|_{L^2_a}},
\qquad
e_{H^1}=\frac{\|u_h-u_{\mathrm{ms},\ell}\|_{a}}{\|u_h\|_{a}},
\]
where $\|v\|^2_{L^2_a}\coloneqq (av,v)_{\Omega}$ and $\|v\|^2_{a}\coloneqq (a\nabla v,\nabla v)_{\Omega}$ for any $v\in H^1(\Omega)$. 
Figure~\ref{fig:darcy-snapshot-combined} compares the fine-grid solution with a representative multiscale solution using $\ell=2$ (so $n_i 2^{\ell}=16$) and $\delta/h=2$. The multiscale field captures the coarse structure of the reference solution, while the absolute-difference panel localizes the remaining discrepancy.
\begin{figure}[htbp]
\centering
\includegraphics[width=0.98\linewidth]{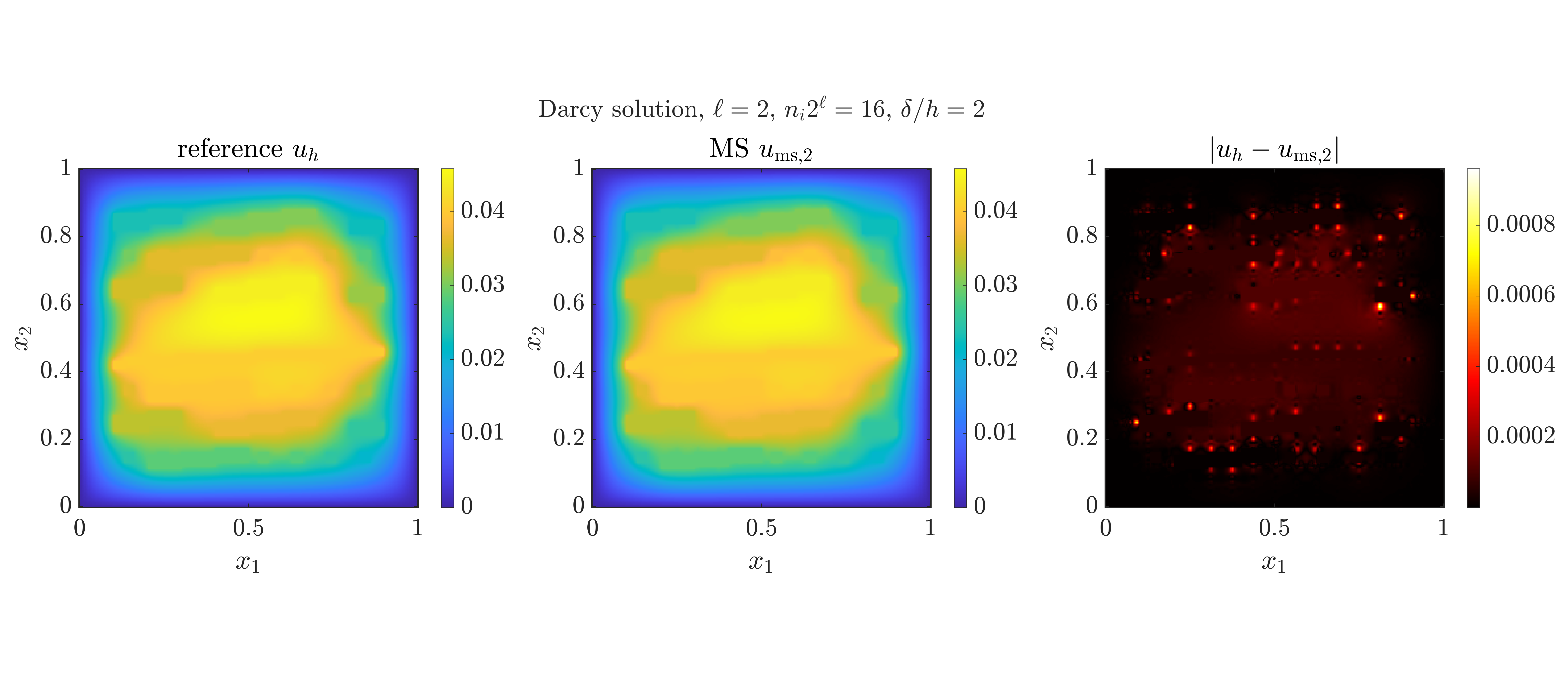}
\caption{Darcy reference solution, representative multiscale solution, and pointwise absolute difference for $\ell=2$ ($n_i 2^{\ell}=16$) and $\delta/h=2$.}
\label{fig:darcy-snapshot-combined}
\end{figure}
Figure~\ref{fig:darcy-errors-combined} plots $\log_{10} e_{L^2}$ and $\log_{10} e_{H^1}$ versus the common overlap width $\delta/h$. The curves show the combined influence of the trace-space level and the overlap width on the accuracy of the multiscale approximation.

\begin{figure}[htbp]
\centering
\includegraphics[width=0.95\linewidth]{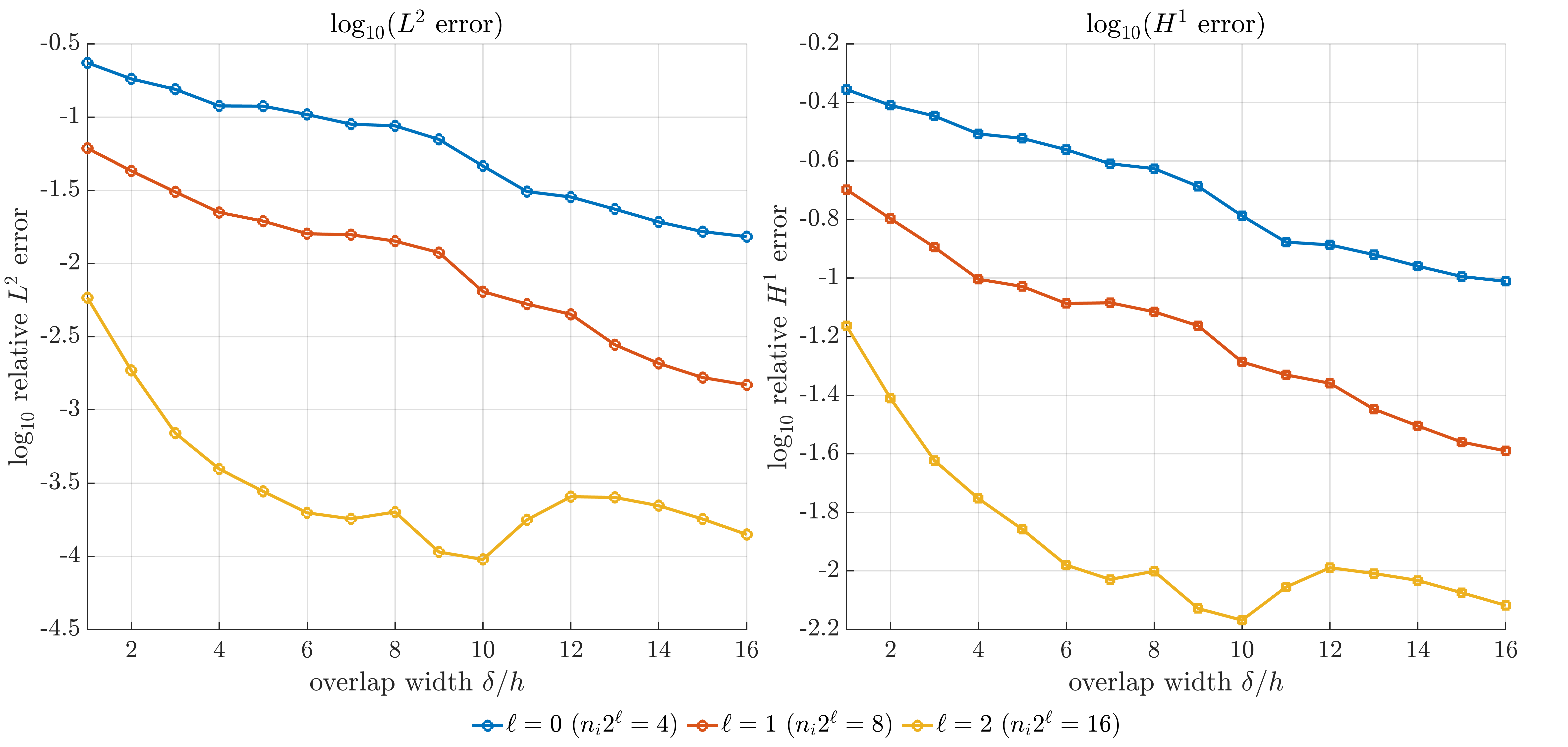}
\caption{Darcy relative errors versus the overlap width $\delta/h$ in logarithmic scale.}
\label{fig:darcy-errors-combined}
\end{figure}

\subsection{Convection-diffusion problem}

The second test is a convection-diffusion problem with homogeneous Dirichlet boundary condition,
\[
-\nabla\cdot(\varepsilon\nabla u)+b\cdot\nabla u = 1
\quad \text{in } \Omega=(0,1)^2,
\qquad
u=0 \quad \text{on } \partial\Omega .
\]
The fine grid contains $512\times512$ elements and the coarse grid contains $16\times16$ blocks; hence $h=1/512$, $H=1/16$, and each coarse block contains $32\times32$ fine elements. The diffusion coefficient is $\varepsilon=1.000\times10^{-2}$. The velocity field is
\[
b_1(x,y)=\alpha\sin(a_0\pi x)\cos(a_0\pi y),
\qquad
b_2(x,y)=-\alpha\cos(a_0\pi x)\sin(a_0\pi y),
\]
with $\alpha=2$ and $a_0=24$. The stabilization parameter is $\tau=h^2/(12\varepsilon)=3.179\times10^{-5}$. With the convention $\mathrm{Pe}_h=\|b\|_{L^\infty}h/(2\varepsilon)$, the maximum cell P\'{e}clet number is $1.952\times10^{-1}$, so the fine grid resolves the diffusion scale for this test.

Figure~\ref{fig:cv-snapshot-combined} shows the fine-grid solution, the representative multiscale solution with $\ell=2$ (so $n_i 2^{\ell}=16$) and $\delta/h=2$, and the corresponding pointwise absolute difference. The comparison visualizes the effect of the reduced wavelet space before aggregating the discrepancy into the global relative error norms.

\begin{figure}[htbp]
\centering
\includegraphics[width=0.98\linewidth]{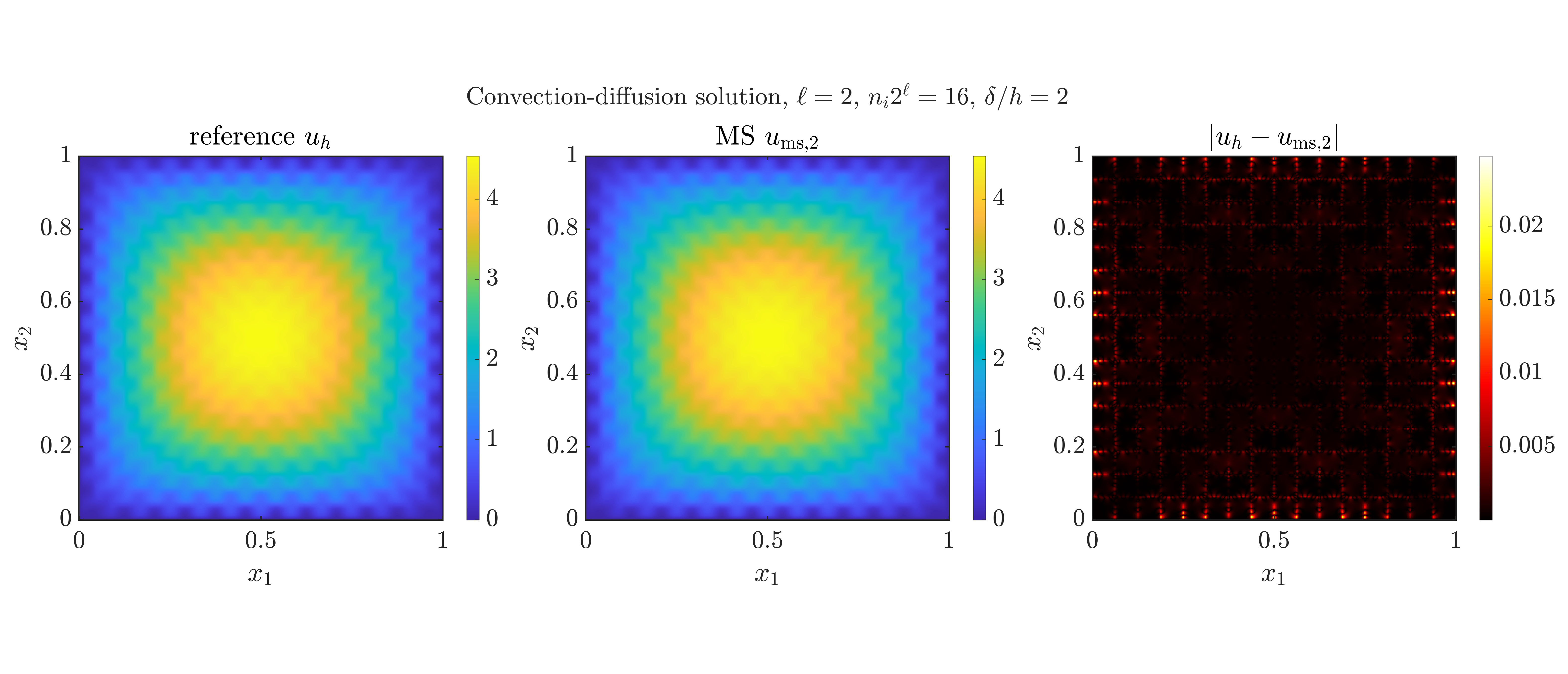}
\caption{Convection-diffusion reference solution, representative multiscale solution, and pointwise absolute difference for $\ell=2$ ($n_i 2^{\ell}=16$) and $\delta/h=2$.}
\label{fig:cv-snapshot-combined}
\end{figure}

Since this operator is nonsymmetric, the reported $H^1$ error is measured with the diffusion stiffness matrix, while the $L^2$ error is measured with the standard mass matrix. Figure~\ref{fig:cv-errors-combined} shows the logarithmic error history for the three levels as the common overlap width is increased.

\begin{figure}[htbp]
\centering
\includegraphics[width=0.95\linewidth]{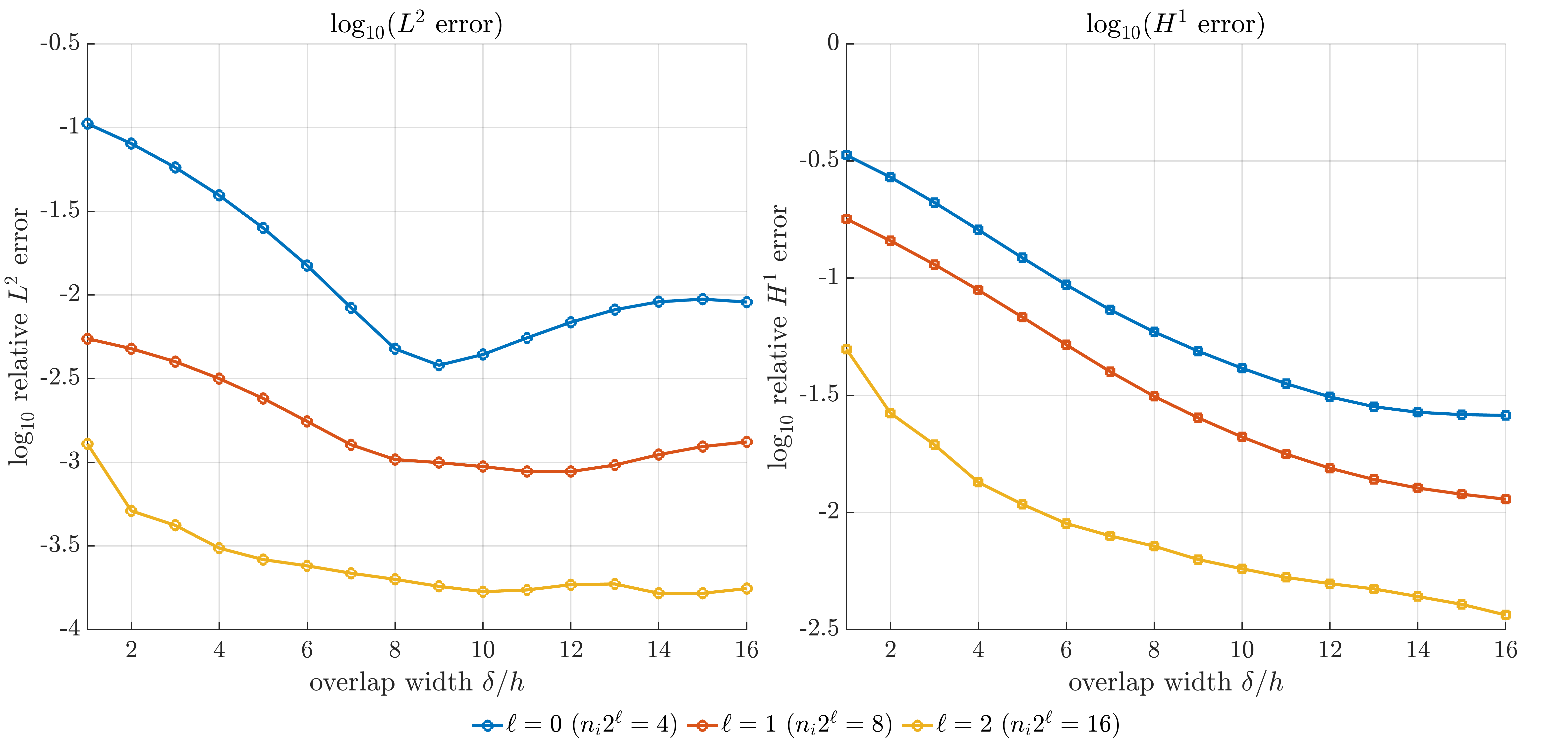}
\caption{Convection-diffusion relative errors versus the overlap width $\delta/h$ in logarithmic scale.}
\label{fig:cv-errors-combined}
\end{figure}

\FloatBarrier
\subsection{Helmholtz problem with first-order absorbing boundary condition}

The third test is the constant-coefficient Helmholtz problem with a localized source and a first-order absorbing boundary condition,
\[
-\Delta u-k^2u=f \quad \text{in } \Omega=(0,1)^2,
\qquad
\partial_n u-\mathrm{i}ku=0 \quad \text{on } \partial\Omega .
\]
The fine grid contains $640\times640$ elements and the coarse grid contains $40\times40$ blocks; hence $h=1/640$, $H=1/40$, and each coarse block contains $16\times16$ fine elements. The computation uses 20 fine-grid points per wavelength, so the wavelength is $20h=1/32$ and the wave number is $k=64\pi$. The source is a Gaussian centered at $(1/2,1/2)$ with support radius $h$.

Onlylevel $\ell=2$ is used for this indefinite wave problem; since $n_i=4$, this gives $n_i 2^{\ell}=16$ local boundary basis functions and the same number of discrete harmonic extensions. This choice keeps the local trace space sufficiently rich and isolates the effect of the overlap width $\delta/h$. The reported $L^2$ error is measured with the standard mass matrix, while the reported $H^1$ error is measured with $H^1(\Omega)$-seminorm, rather than with the indefinite Helmholtz quadratic form.

Figure~\ref{fig:hel-snapshot-combined} compares the real parts of the fine-grid and multiscale Helmholtz fields for the representative choice $\ell=2$ and $\delta/h=2$. The absolute-difference panel is computed from the complex-valued fields, namely $|u_h-u_{\mathrm{ms},2}|$.

\begin{figure}[htbp]
\centering
\includegraphics[width=0.98\linewidth]{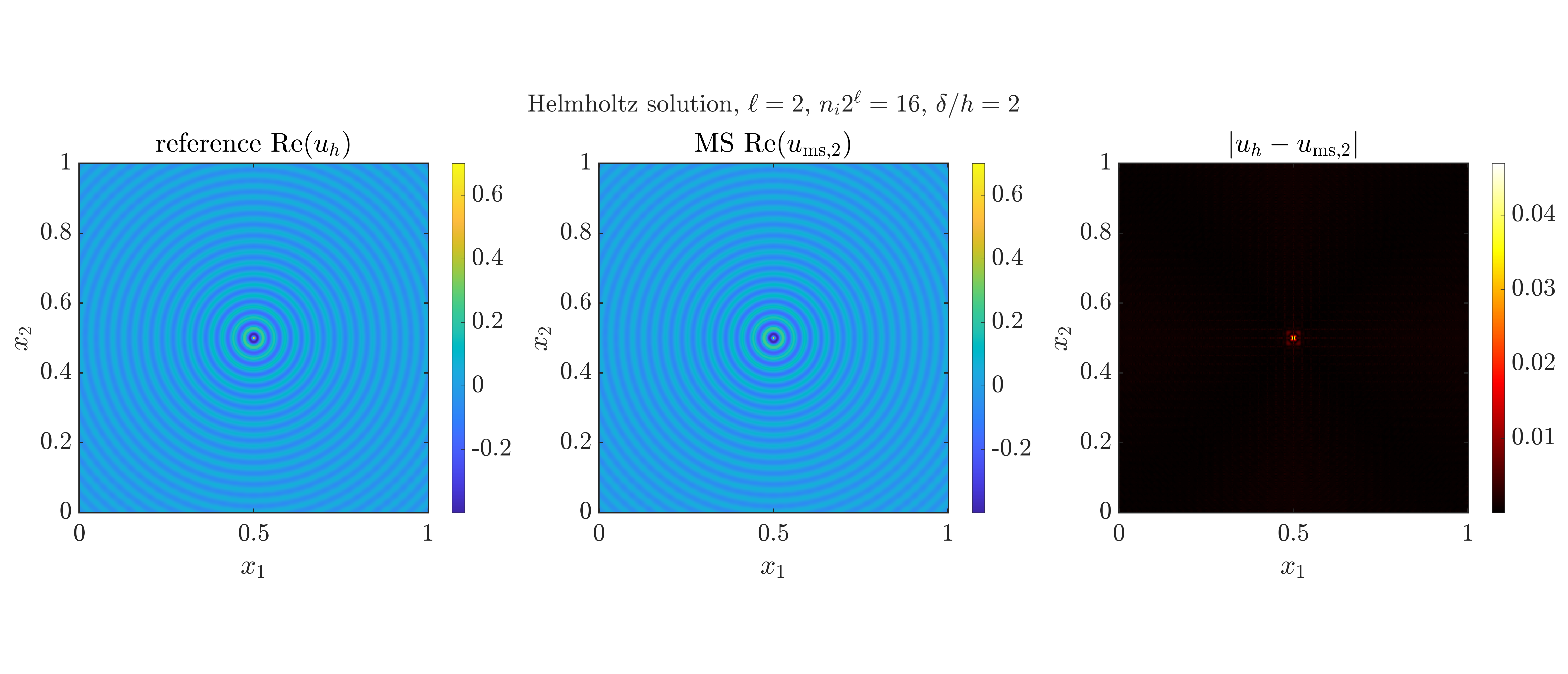}
\caption{Helmholtz reference solution, representative wavelet multiscale solution, and pointwise absolute difference for $\ell=2$ ($n_i 2^{\ell}=16$) and $\delta/h=2$. The first two panels show real parts, while the difference is the complex modulus.}
\label{fig:hel-snapshot-combined}
\end{figure}

Figure~\ref{fig:hel-errors-combined} displays the two relative error curves in the same logarithmic axes. The figure should be interpreted as the error of the multiscale approximation relative to the fine-grid Helmholtz solution; it does not remove the dispersion error already present in the fine-grid reference solution.

\begin{figure}[htbp]
\centering
\includegraphics[width=0.95\linewidth]{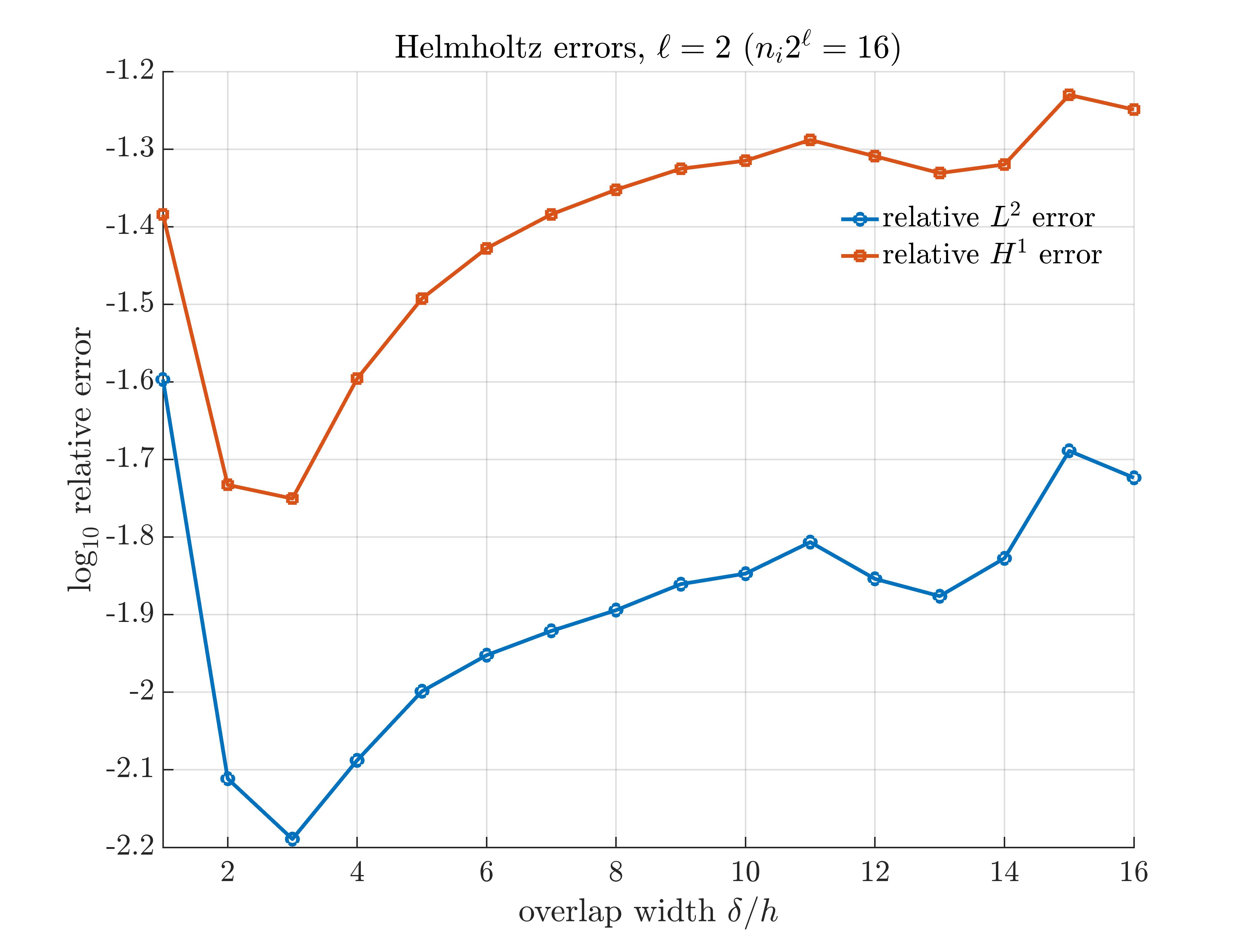}
\caption{Helmholtz relative $L^2$ and $H^1$ errors versus the overlap width $\delta/h$ in logarithmic scale for $\ell=2$ ($n_i 2^{\ell}=16$).}
\label{fig:hel-errors-combined}
\end{figure}
\section{Conclusion and future work}\label{sec:conclusion}
In this paper, we present the Edge Multiscale Finite Element Methods (EMsFEM) in the discrete setting and outline the main ideas behind their error estimates. Under the current methodology, which employs partition of unity functions, the error estimates depend on two key components:
\begin{enumerate}
    \item An $L^2$ \textit{a priori} estimate for discrete generalized harmonic functions in each subdomain.
    \item A Caccioppoli-type estimate for discrete generalized harmonic functions in each subdomain.
\end{enumerate}
We present several numerical tests, including Darcy problems with highly oscillatory coefficients, convection-dominated diffusion equations, and Helmholtz problems with large wavenumbers, to illustrate the performance of the method. Additionally, several questions remain open for future investigation:
\begin{itemize}
    \item[1.] EMsFEMs are constructed using local multiscale basis functions defined by Dirichlet boundary conditions. Could other boundary conditions—such as Neumann or, more generally, Robin-type conditions—be used instead, particularly in cases where Dirichlet conditions lead to ill-posedness?
    \item[2.] Is it possible to remove the partition of unity functions in the construction of multiscale basis functions? While this may compromise the conformity of the numerical scheme, it could help preserve desirable properties—such as divergence in Maxwell problems—and potentially improve the approximation properties of the multiscale basis functions by eliminating the overlapping parameter in the analysis.
\end{itemize}
\bibliographystyle{abbrv}
\bibliography{reference}
\end{document}